\documentclass[11pt]{amsart}

\usepackage{amsmath}
\usepackage{amssymb}
\usepackage{graphicx}

\newtheorem{theorem}{Theorem}[section]
\newtheorem{corollary}[theorem]{Corollary}
\newtheorem{lemma}[theorem]{Lemma}

\theoremstyle{definition}
\newtheorem{definition}[theorem]{Definition}
\newtheorem{example}[theorem]{Example}
\newtheorem{remark}[theorem]{Remark}

\numberwithin{equation}{section}

\title[Nonlinear Spectral Resolution]{Nonlinear Spectral Resolution}

\author{Wen Hsiang Wei}

\address{Department of Statistics, Tung Hai University, Taiwan}
\email{wenwei@thu.edu.tw}

\begin{document}

\begin{abstract}

A quasi-product on the normed space is defined. In addition, the notions of the
eigenvectors of a linear operator can be extended for the nonlinear operator.
Based on the quasi-product and the generalized eigenvectors, the spectral
theorems for certain possibly nonlinear operators which include bounded linear symmetric
operators as special cases can be proved. The operational calculus of
the classes of possibly nonlinear operators is developed and nonlinear spectral
operators are given.

\end{abstract}

\maketitle


\section{Introduction}

Operator theory has been at the heart of research in analysis (see
\cite{abramovich}; \cite{pier}, Chapter 4). Moreover, as implied by
\cite{neuberger}, considering nonlinear case should be
essential. Developing useful results for the operators holds the promise for the
wide applications of nonlinear functional analysis to a variety of scientific
areas.

In classical functional analysis, the space of bounded linear operators is a
normed space endowed with a sensible norm. Further, by defining the
composition of two bounded linear operators as the operation of
multiplication, the space of the bounded linear operators is also a normed algebra.
In Section 2, a normed function is defined and some set of possibly
nonlinear operators from a normed space into a normed space turns out to be
a normed space. Further, if the domain and the range of the possibly nonlinear
operators are the same normed algebra, the normed space of the possibly nonlinear
operators can be a normed algebra by defining an operation of multiplication
for two operators.

Spectral theory is one of the main topics of modern functional analysis and its
applications (see \cite{kreyszig}; \cite{zeidler}). Spectral theory for certain
classes of linear operators has been well developed (see \cite{dunford}),
particularly linear symmetric operators in a
Hilbert space. Spectral theory for the nonlinear operators is an emerging field
in functional analysis (see
\cite{appell}). However, relatively little has been done for the spectral resolution of
the possibly nonlinear operator of interest and the associated extensions
which are the main objectives of this article. The inner product is important for
the development of spectral theory in Hilbert spaces. Therefore, in order to
develop the spectral decomposition of the possibly nonlinear operator on the normed
space, a "product" which is referred to as the quasi-product and plays the role
analogous to the inner product in the inner product space is proposed in next
section. Section 3 defines a generalized real definite operator which can be
considered as the generalization of the linear symmetric operator on the
Hilbert space. Furthermore, the generalized eigenvalue which can be
considered as the generalized version of the eigenvalue in the linear case is also
defined. The spectral theorems of the bounded generalized real definite
operator are given in the section. Finally, the extensions of the results in Section 3,
including the operational calculus of the generalized real definite
operator and the nonlinear spectral operators (see \cite{dunford2}), are given in Section 4.
Note that the results in this article can be used to prove the spectral theorems for
a more general class of possibly nonlinear operators which includes bounded linear normal
operators as special cases. In addition, the spectral representations for the unbounded
nonlinear operators can be also proved based on these results.

Hereafter \( D(F) \) and \( R(F) \) are denoted as the domain and the range of
an operator \( F
\), respectively, and the notation \( ||\cdot||_{Z} \) is denoted as the norm of the
normed space \( Z \). The space of interest is the normed space implicitly. On
the other hand, the Banach space or the normed algebra will be indicated
explicitly. Note that the vector spaces and the normed spaces of interest in this
article are not trivial, i.e., not only including the zero element. In this article,  only the proofs of the main results are given.
The proofs of other lemmas, theorems, and corollaries are delegated to the
supplementary materials which can be found at \\
\centerline{\bf http://web.thu.edu.tw/wenwei/www/supplement.pdf/}

\section{Nonlinear functional spaces}

\subsection{Basics}

Let \( X \) and \( Y \) be the normed spaces over the field \( K \) with some
sensible norms, where \( K \) is either a real field \( R \) or a complex field \( C
\). Let
\( V(X,Y) \) be the set of all operators from \( X  \) into \( Y \), i.e.,
the set of arbitrary maps from \( X \) into \( Y \).
Let the algebraic operations of \( F_{1}, F_{2} \in V(X,Y) \) be the operators with
\( (F_{1}+F_{2})(x)=F_{1}(x)+F_{2}(x) \) and
\( (\alpha F_{1})(x)=\alpha F_{1}(x) \) for \( x \in X \), where \( \alpha \in K \) is a
scalar. Also let the zero element in
\( V(X,Y) \) be the operator with the image equal to the zero element in \( Y \).

In this subsection, the basic properties of the nonlinear functional spaces are
given. The proofs of several theorems and corollaries, including Theorem~\ref{thm:2},
Corollary~\ref{cor:1}, Theorem~\ref{thm:3}, Corollary~\ref{cor:2}, and Theorem~\ref{thm:5}, are quite routine and are delegated to
the supplementary materials.

\begin{theorem}\label{thm:1}

\( V(X,Y) \) is a vector space over the field \( K \).

\end{theorem}

The routine proof of the above theorem is not presented. Define a
non-negative extended real-valued function
\( p \), i.e., the range of \( p \) including \( \infty \), on \( V(X,Y) \) by
\begin{eqnarray*}
p(F)=\max\left(\sup_{x \neq 0,x \in X} \frac{\|F(x)\|_{Y}}{\|x\|_{X}},
\|F(0)\|_{Y}\right)
\end{eqnarray*}
for \( F \in V(X,Y) \). The non-negative extended real-valued function \( p \) is a
generalization of the norm for the linear operators. Let
\( B(X,Y) \), the subset of \( V(X,Y) \), consist of all operators with
\( p(F) \) being finite.

\begin{theorem}\label{thm:2}

\( p \) is a norm on \( B(X,Y) \) and \( [B(X,Y),p] \) is a normed space.

\end{theorem}

Hereafter the norm \( p \) is used, i.e., \( ||F||_{B(X,Y)}=p(F) \) for \( F \in
B(X,Y)
\). Note that the bounded linear operators fall in
\( B(X,Y) \). As \( X=Y \), the notation
\( B(X)=B(X,X) \) is used.

Let the notation of the composition of two operators be \( \circ \) hereafter.

\begin{corollary}\label{cor:1}

Let \( F_{1} \in B(X,Y) \) and \( F_{2} \in B(Y,Z) \), where \( Z \) is a normed
space. Then
\begin{eqnarray*}
\|F_{1}(x)\|_{Y} \leq \|F_{1}\|_{B(X,Y)}\|x\|_{X},x \neq 0.
\end{eqnarray*}
 If \( F_{1}(0)=F_{2}(0)=0 \), then
\begin{eqnarray*}
\|F_{2} \circ F_{1}\|_{B(X,Z)} \leq \|F_{1}\|_{B(X,Y)}\|F_{2}\|_{B(Y,Z)}.
\end{eqnarray*}

\end{corollary}

Let \( BC(X,Y) \), the subspace of \( B(X,Y )\), consist of the bounded continuous
operators. It is well known that the space of all bounded linear operators from
a normed space X to a Banach space Y is complete. The analogous results also
hold for the spaces
\( B(X,Y) \) and \( BC(X,Y) \).

\begin{theorem}\label{thm:3}

If \( Y \) is a Banach space, then \( B(X,Y)\) and \( BC(X,Y) \) are Banach spaces.

\end{theorem}

Unlike a linear operator, a continuous nonlinear operator \( F \) might not be
bounded. The following corollary gives sufficient conditions for the
boundedness of a continuous operator.

\begin{corollary}\label{cor:2}

Let \( \mathcal{K} \) be a compact subset of \( X \). Let \( F: X \rightarrow Y \)
be continuous on
\( \mathcal{K} \).
\( F \in B(X,Y) \) if the following conditions hold: \\
(a) \( \|F(\mathcal{K}^{c})\|_{B(X,Y)} \leq M \), where \( \mathcal{K}^{c} \) is the complement of the set  \( \mathcal{K} \) and
\( M \) is a positive
number. \\
(b) \( \lim_{x \rightarrow 0}\|F(x)\|_{Y}/\|x\|_{X} \) exists and is
finite.

\end{corollary}

The linear functionals in the dual space of a normed space can distinguish the
points of the normed space. Since the dual space of the normed space \( X \) is
a subspace of \( B(X,K) \), the results holds for the nonlinear functionals in
\( B(X,K) \). Further, the following theorem gives the counterpart of the one for the bounded
linear functional in the second dual space.

\begin{theorem}\label{thm:4}

Let \( x \neq 0 \) and \( g_{x}: B(X,K) \rightarrow K \) be defined by \(
g_{x}(F)=F(x)
\) for \( F
\in B(X,K) \). Then \( g_{x} \in B[ B(X,K),K] \) and
\begin{eqnarray*}
\left|\left|x\right|\right|_{X}=\sup_{F \neq 0,F \in B(X,K)}\frac{|F(x)|}{\|F\|_{B(X,K)}}
=\left|\left|g_{x}\right|\right|_{B[B(X,K),K]}.
\end{eqnarray*}

\end{theorem}

\begin{proof} Since all bounded linear functionals fall in
\( B(X,K) \), then for any \( x \neq 0 \) there exists a bounded linear
functional \( F_{x} \in B(X,K) \) such that
\( ||F_{x}||_{B(X,K)}=1 \) and \( F_{x}(x)=||x||_{X} \) . Thus,
\begin{eqnarray*}
\sup_{F \neq 0,F \in B(X,K)}\frac{\left|F(x)\right|}{\|F\|_{B(X,K)}}
\geq \frac{\left|F_{x}(x)\right|}{\|F_{x}\|_{B(X,K)}}=\|x\|_{X}.
\end{eqnarray*}
On the other hand, \( ||x||_{X} \geq \sup_{F \neq 0,F \in B(X,K)}
\left|F(x)\right|/\|F\|_{B(X,K)} \) because by Corollary~\ref{cor:1},
\( |F(x)| \leq ||F||_{B(X,K)}||x||_{X} \). Finally,
\begin{eqnarray*}
\left|\left|g_{x}\right|\right|_{B[B(X,K),K]}
=\sup_{F \neq 0,F \in B(X,K)}\frac{|F(x)|}{\|F\|_{B(X,K)}}
=\left|\left|x\right|\right|_{X}.
\end{eqnarray*}
because \( g_{x}(0)=0 \).
\end{proof}

The Banach space of all bounded linear operators on a Banach space \( X
\) is a Banach algebra with the multiplication being the composition of the
operators. The Banach space \( B(X) \) is not a Banach algebra with the
multiplication being the composition of the operators and \( X \) being a Banach
space. However,
\( B(X) \) can be a normed algebra or a Banach algebra depending on
\( X \) being a normed algebra or a Banach algebra (see Theorem~\ref{thm:3}) as the multiplication
of the operators is defined properly.

\begin{theorem}\label{thm:5}

Let \( X \) be a normed (Banach) algebra. Define the multiplication of \( F_{1} \)
and
\( F_{2}
\) in
\( V(X,X) \) by
\begin{eqnarray*}
(F_{1}*F_{2})(x)=\frac{F_{1}(x)F_{2}(x)}{\|x\|_{X}},x\neq 0,x \in X,
\end{eqnarray*}
and
\begin{eqnarray*}
(F_{1}*F_{2})(0)=F_{1}(0)F_{2}(0).
\end{eqnarray*}
Then \( B(X) \) is a normed (Banach) algebra. As \( X \) has a unit,
\( B(X) \) is a normed (Banach) algebra with a unit \( e \).

\end{theorem}

\subsection{Quasi-product}

\begin{definition}\label{def:1}

Let \( X \) be a normed space and \( S \) be a subset of \( X \). A quasi-product
\( [\cdot,\cdot]_{S} \) on \( S \) is a mapping (or a map) of \( S \times S \) into the scalar field \( K \)
with the following properties: \\ (a)
\begin{eqnarray*}
[x,x]_{S} \geq 0
\end{eqnarray*}
for \( x \in S \). \\ (b)
\begin{eqnarray*}
\left|[x,y]_{S}\right| \leq \overline{c}\left|\left|x\right|\right|_{X}
\left|\left|y\right|\right|_{X}
\end{eqnarray*}
for \( x,y \in S \), where \( \overline{c} \) is a positive number. \\ (c)
\begin{eqnarray*}
\left[\sum_{i=1}^{n}\alpha_{i} x_{i},y\right]_{S}=c(y) \sum_{i=1}^{n}
\alpha_{i}\left[ x_{i},y\right]_{S}
\end{eqnarray*}
for any \( n \geq 1 \), \( x_{1}, \ldots,x_{n}, y, \sum_{i=1}^{n}\alpha_{i} x_{i}
\in S \) and \( \alpha_{1}, \ldots, \alpha_{n} \in K \), where \( c: S \rightarrow R \)
is a positive bounded function and is bounded away from 0.

A quasi-product is quasi-symmetric if and only if
\( [x,y]_{S}=q(x,y)\overline{[y,x]}_{S} \), where
\( \overline{z} \) is the conjugate of the complex number \( z \) and
\( q:S \times S \rightarrow R \) is a positive bounded function and
is bounded away from 0. A quasi-symmetric quasi-product is symmetric if and
only if \( [x,y]_{S} \) is equal to the conjugate of \( [x,y]_{S} \), i.e., \(
[x,y]_{S}=\overline{[y,x]}_{S} \).

\end{definition}

\begin{remark}\label{rem:1}

As \( S=X \), \( c(y)=\overline{c}=1 \),
\( [x,x]_{X}=0 \) implies \( x=0 \), and
the quasi-product is symmetric,
\( X \) is an inner product space with the inner product being
the quasi-product. On the other hand,
an inner product \( <\cdot,\cdot>_{X} \) on the inner product space \( X \) can be a symmetric
quasi-product by setting \( <x,y>_{X}=[x,y]_{X} \).

\end{remark}

The following are examples of the quasi-products on the normed spaces or the
subset of some normed space.

\begin{example}

Let \( X \) be an inner product space with an inner product \( <\cdot,\cdot>_{X}
\) and the norm induced by the inner product. Then
\begin{eqnarray*}
[x,y]_{X}=c(y)<x,y>_{X}
\end{eqnarray*}
for \( x,y \in X \), where \( c \) is a positive bounded function on \( X \) and is bounded
away from 0, for example,
\begin{eqnarray*}
c(y)=\frac{\left|\left|y\right|\right|_{X}} {\left|\left|y\right|\right|_{X}+1}+k
\end{eqnarray*}
and
\( k > 0 \).

\end{example}

\begin{example}

Let \( X \) be the real normed space of real-valued Lebesgue integrable
functions on the measurable set \( \Omega \subset R \) with the norm
\( ||x||_{X}=\int_{\Omega}|x|d\mu \) for \( x \in X \),
where \( \mu \) is the Lebesgue measure. Let the symmetric quasi-product
defined by
\begin{eqnarray*}
[x,y]_{X}=\int_{\Omega}xd\mu \int_{\Omega}yd\mu
\end{eqnarray*}
for \( x,y \in X \).

\end{example}

\begin{example}

Let \( X \) be the real normed space of real-valued bounded functions on the
compact domain \(
\Omega \in R \) with the supremum norm and the subset \( S \) of \( X \) consist
of the bounded Lebesgue integrable functions with positive integrals. Define
the non-symmetric quasi-product on \( S \) by
\begin{eqnarray*}
[x,y]_{S}=\int_{\Omega}xd\mu \sup_{t \in \Omega}y(t),
\end{eqnarray*}
where \( \mu \) is the Lebesgue measure.

\end{example}

Note that the quasi-product function \( [x,y]_{X}=f_{y}(x) \), considered as a
function of
\( x \), is continuous on \( X \). In addition,
if the quasi-product is symmetric, it is jointly continuous.

The bilinear form on a Hilbert space has a representation associated with the
continuous linear operator (see \cite{kreyszig}, Theorem 3.8-4; \cite{taylor},
Chapter VI, Theorem 1.2). The following theorem can be considered as the
nonlinear counterpart of the linear case.

\begin{theorem}\label{thm:6}

Let \( F \in B(X,Y) \) with \( F(0)=0 \) and there exists a positive function \( d
\) defined on \( Y \) and being bounded away from 0 such that
\( [y,y]_{Y} =d(y)||y||^{2}_{Y} \) for \( y \in Y \). Then \( h: X \times Y \rightarrow K \)
has a unique representation
\begin{eqnarray*}
h(x,y)=[F(x),y]_{Y}
\end{eqnarray*}
if and only if the function \( h \) has the following properties: \\ (a)
\begin{eqnarray*}
\left|h(x,y)\right| \leq
\overline{c} \left|\left|x\right|\right|_{X}\left|\left|y\right|\right|_{Y}
\end{eqnarray*}
for \( x \in X \) and \( y \in Y \), where \( \overline{c} \) is a positive number. \\
(b) For \( y \in Y \), there exists a \( z \in Y \) such that
\begin{eqnarray*}
h(x,y)=[z,y]_{Y}
\end{eqnarray*}
for any \( x \in X \).

\end{theorem}

\begin{proof} The "only if" part is obvious. To prove "if" part,
define \( F(x)=z \) with \( F(0)=0 \). \( F \) is well-defined since for \( x_{1}=x_{2}
\),
\begin{eqnarray*}
h(x_{1},y)=[z,y]_{Y}=h(x_{2},y)
\end{eqnarray*}
and thus \( F(x_{1})=F(x_{2})=z \). Further, for \( F \neq 0 \),
\( F \in B(X,Y) \) because there exists \( c_{1} > 0 \) such that
\begin{eqnarray*}
&&\left|\left|F\right|\right|_{B(X,Y)} \\ &=&\sup_{x \neq 0,x \in X,F(x) \neq 0}
\frac{\left|\left|F(x)\right|\right|^{2}_{Y}}
{\left|\left|x\right|\right|_{X}\left|\left|F(x)\right|\right|_{Y}} \\ &\leq& c_{1}
\sup_{x \neq 0,x \in X,F(x)\neq 0}\frac{\left|[F(x),F(x)]_{Y}\right|}
{\left|\left|x\right|\right|_{X}\left|\left|F(x)\right|\right|_{Y}} \\ &\leq& c_{1}
\sup_{x \neq 0,x \in X,y\neq 0,y \in Y}\frac{\left|[F(x),y]_{Y}\right|}
{\left|\left|x\right|\right|_{X}\left|\left|y\right|\right|_{Y}} \\ &\leq
&c_{1}\overline{c}.
\end{eqnarray*}
To prove the uniqueness of \( F \), let \( h(x,y)=[F(x),y]_{Y}=[G(x),y]_{Y} \) for
\( x \in X \) and \( y \in Y \). Then by the properties of the quasi-product,
\begin{eqnarray*}
\left|[F(x),y]_{Y}-[G(x),y]_{Y}\right| \geq c_{2} \left|[F(x)-G(x),y]_{Y}\right|
\end{eqnarray*}
implies \( [F(x)-G(x),y]_{Y}=0 \) for any \( y \in Y \) and thus
\begin{eqnarray*}
c_{3}\left|\left|F(x)-G(x)\right|\right|^{2}_{Y} \leq [F(x)-G(x),F(x)-G(x)]_{Y}=0,
\end{eqnarray*}
where \( c_{2} \) depending on \( y \) and \( c_{3} \) depending on \( x \) are
both positive numbers. Therefore, \( F(x)=G(x) \).
\end{proof}

\section{Generalized real definite operators}

The goal of this section is to formulate the spectral resolutions of some class of possibly
nonlinear operators. In this section and the following subsection, i.e., Section 3
and Section 4.1, suppose that the operator \( F: X
\rightarrow X \) of interest satisfies \( F(0)=0 \) and the quasi-products
on \( X \) exist. For
\( \tilde{F}: X \rightarrow X \) with \( \tilde{F}(0) \neq 0 \),
the corresponding spectral resolution can be obtained by the shift \(
\tilde{F}(x)=F(x)+s(x), x \in X \), where \( s: X \rightarrow X \) satisfies \( s(0)=\tilde{F}(0) \),
for example, \( s(x)=\tilde{F}(0) \) (also see Remark~\ref{rem:6} in Section 4.1).

The spectral resolution of the bounded linear symmetric operators, due to
Hilbert, is a limit of Riemann-Stieltjes sums in the sense of operator
convergence (also see
\cite{lax}, Chapter 31), i.e., having the form of Riemann-Stieltjes integral.
The spectral resolutions of the generalized real definite operators defined in this
section also have the form of Riemann-Stieltjes integral in terms of the quasi-product or
 in the sense of operator convergence as certain classes of quasi-products are employed.
For the bounded linear symmetric operators, the spectral resolution involves
both linear projection operators and linear positive operators. The nonlinear
counterparts of the linear positive operators and the linear projection
operators are defined and their basic properties are given in the first two
subsections. Then, the spectral resolutions
are given in the last subsection.

\subsection{G-positive operators} Linear spectral theory for the symmetric
operators involves the positive operators. Similarly, for the nonlinear
operators with real spectrums, the spectral theorems also involve the positive
operators. The main theorem in this subsection, Theorem~\ref{thm:7}, indicates that the
multiplication of two positive operators remains positive given some sufficient
conditions imposed on the quasi-product. Based on the theorem,  two
corollaries give the existence and the uniqueness of the square root of the
positive operator. Furthermore, the spectral theorem in
Section 3.3 can be proved by using the theorem.
Note that the proofs of Theorem~\ref{thm:7}, Corollary~\ref{cor:3}, and
Corollary~\ref{cor:4}  in
this subsection are
delegated to the supplementary materials.

The nonlinear positive operator is defined as follows.

\begin{definition}\label{def:2}

Let \( X \) be a normed space. An operator
\( F: X \rightarrow X \) is generalized real definite if and only if there exist
a quasi-product \( [\cdot,\cdot]_{X} \) and an operator
\( g: X
\rightarrow X
\)  satisfying
\( g(x) \neq 0 \) for \( x \neq 0 \) and
\begin{eqnarray*}
[F(x),g(x)]_{X} \in R
\end{eqnarray*}
for \( x \in X \). Furthermore, \( F \) is g-positive, denoted by \( F \geq 0 \), if
and only if
\begin{eqnarray*}
[F(x),g(x)]_{X} \geq 0
\end{eqnarray*}
for \( x \in X \). For the operators \( F_{1}:X \rightarrow X \) and \( F_{2}:X \rightarrow X \),
\begin{eqnarray*}
F_{1} \geq F_{2}
\end{eqnarray*}
if and only if
\begin{eqnarray*}
F_{1}-F_{2} \geq 0.
\end{eqnarray*}
As \( X \) is a normed algebra, a g-positive operator \( F \) has a square root if
and only if there exists an operator \( G: X \rightarrow X \) such that
\( [G(x)]^{2}=G(x)G(x)=F(x) \) for \( x \in X \). \( G \) is then called a square root of \( F \).

\end{definition}

 As \( F=0 \), \( F \) is g-positive. \( G=0 \) is the square root of \( F=0 \) defined
on the normed algebra \( X \).

\begin{remark}\label{rem:2}

The generalized real definiteness of the operator \( F \) relies on both the
operator \( g \) and the quasi-product. As \( F \) is a linear symmetric operator
on a Hilbert space, the operator
\( g \) is the identity map, and the quasi-product is the inner product,
\( F \) is generalized real definite and
\( F \) being positive  implies \( F \) being g-positive. Therefore,
the g-positivity extends the notion of the positivity.

As \( F: D(F) \rightarrow Y \) and \( g: D(F)
\rightarrow Y \) satisfies
\( g(x) \neq 0 \) for \( x \neq 0 \) and \( x \in D(F) \),
the above definition can be modified and the associated expressions are
\( [F(x),g(x)]_{Y} \in R \) and \( [F(x),g(x)]_{Y} \geq 0 \) for \( x \in D(F) \), where
\( D(F) \) is a subset of \( X \).

\end{remark}

Some basic properties of the positive operators are given by the following
lemma. The routine proofs are not presented.

\begin{lemma}\label{lem:1}

Let \( X \) be a normed space and
\( F_{1}, F_{2}, G_{1}, G_{2} \) be the operators from \( X \) into \( X \).
\\
(a)
\( F_{1} \geq F_{2} \) if and only if
\( [F_{1}(x),g(x)]_{X} \geq [F_{2}(x),g(x)]_{X} \) for
\( x \in X \). \\
(b) If \( \alpha \in R \),
\( \alpha \geq 0 \), and \( F_{1} \geq 0 \), then
\( \alpha F_{1} \geq 0 \). \\
(c) If \( F_{1} \geq 0 \) and \( F_{2} \geq 0 \), then
\( F_{1}+F_{2} \geq 0 \). \\
(d) If \( F_{1} \leq G_{1} \) and \( F_{2} \leq G_{2} \), then \( F_{1}+F_{2}
\leq G_{1}+G_{2} \).

\end{lemma}

For a linear symmetric operator \( F \) on a Hilbert space \( X \), the normed
value of \( F(x) \) and the associated inner product satisfy the equation
\( <(F\circ F)(x),x>_{X}=||F(x)||_{X}^{2} \). It turns out that the analogous
equation (inequality) defined below plays the key role in the development of
the spectral resolutions of the possibly nonlinear operators.

\begin{definition}\label{def:3}

Let \( X \) be a normed space and \( g: X \rightarrow X \)
satisfying \( g(x) \neq 0 \) for \( x \neq 0 \). For \( x, y \in X \),
a quasi-product has a left integral domain
if and only if for \( g(x) \neq 0 \), \( [y,g(x)]_{X}=0 \) implies \( y=0 \). Further,
as \( X \) is a normed algebra with a unit \( 1 \), for
\( x, y_{1}, y_{2} \in X \), a
quasi-product preserves the positivity if and only if
\( [y_{1}y_{2},g(x)]_{X} \geq 0 \) as \( [y_{1},g(x)]_{X} \geq 0 \) and
\( [y_{2},g(x)]_{X} \geq 0 \). In addition,
a quasi-product is square bounded below if and only if  for
\( x, y \in X \), there exists a positive
number
\(
\underline{k}
\) such that
\( [y^{2},g(x)]_{X} \geq \underline{k} ||y||_{X}^{2}||g(x)||_{X} \)
as \( y=1 \) or \( [y,g(x)]_{X} \in R \).

\end{definition}

The properties of the quasi-product in Definition~\ref{def:3} rely on the operator
\( g
\) and \( g \) is assumed to be the same as the one corresponding to
the generalized real definite operators of interest (see Definition~\ref{def:2}) hereafter.
Note that the quasi-product has these properties on the set
\( X \times R(g) \) by the above definition and thus another
way is to define these properties on the subsets of \( X \times X \). However,
only the set \( X \times R(g) \) is of interest in this article and hence Definition~\ref{def:3} serves the purpose.

In the following, the operators involved in the spectral theorems of the
generalized real definite operators, the positive and negative parts of the
operator, are defined.

\begin{definition}\label{def:4}

Let \( X \) be a normed space and \( F \) be a generalized real definite operator.
The operator \( |F| \) is defined by
\begin{eqnarray*}
\left|F\right|(x)=F(x)
\end{eqnarray*}
if \( [F(x),g(x)]_{X} \geq 0 \) and
\begin{eqnarray*}
\left|F\right|(x)=-F(x)
\end{eqnarray*}
if \( [F(x),g(x)]_{X} < 0 \) for \( x \in X \). The positive part of \( F \) is
\begin{eqnarray*}
F^{+}=\frac{\left|F\right|+F}{2}
\end{eqnarray*}
and the negative part of \( F \) is
\begin{eqnarray*}
F^{-}=\frac{\left|F\right|-F}{2}.
\end{eqnarray*}

\end{definition}

A direct check gives the following lemma.

\begin{lemma}\label{lem:2}
 Let \( F \) be generalized real definite.   \\
(a) Let \( X \) be a normed space. Then \( |F| \geq 0 \). \\ (b) Let \( X \) be a
unital normed algebra. Then \( |F|^{2}=F^{2}
\), where for \( x
\in X
\),
\( |F|^{2}
\) is defined by \( |F|^{2}(x)=|F|(x)|F|(x) \) and
\( F^{2} \) is defined by \( F^{2}(x)=F(x)F(x) \). Furthermore, if
the quasi-product is square bounded below or preserves the positivity, then \(
|F|
\) is a square root of
\( F^{2} \).

\end{lemma}

The above lemma implies that the quasi-product being square bounded below
or preserving the positivity is a sensible condition. Since otherwise, the square
of a generalized real definite operator or even a g-positive operator might not
be g-positive.

It is natural to ask when the g-positive operator has a square root. It turns out
that the quasi-product having a left integral domain or preserving the positivity
is a key sufficient condition. The following theorem indicates that the pointwise
multiplication of two g-positive operators is g-positive given some sufficient
condition.

\begin{theorem}\label{thm:7}

Let \( X \) be a unital Banach algebra,
\( F \geq 0 \) and \( H \geq 0 \). If there exists a
quasi-product (not necessarily square bounded below) preserving the
positivity with
\( X \) being not necessarily commutative or a square bounded below quasi-product
having a left integral domain with \( X \) being commutative, then \( FH \geq 0
\), where \( (FH)(x)=F(x)H(x) \) for \( x \in X \).

\end{theorem}

For a positive linear symmetric operator, the positive square root of the
operator is uniquely determined. In the following, the sufficient conditions for
the existence and the uniqueness of the positive square root of the g-positive
operator are given.

\begin{corollary}\label{cor:3}

Let \( X \) be a unital Banach algebra. \\ (a) If \( G \), the square root of the
g-positive operator \( F \), exists and
\( G
\geq 0
\), the commutative unital Banach algebra \( X \) is an integral domain,
and the quasi-product has a left integral domain, then \( G \) is unique and is
denoted by \( G=F^{1/2} \). \\ (b) If there exists a square bounded below
quasi-product preserving the positivity on
\( X \) (not necessarily commutative) or a square bounded below quasi-product
having a left integral domain with \( X \) being commutative, \( 0 \leq F \leq
1_{X} \), then \( F
\) has a square root \( G \geq 0 \) and
\( G \) commutes with any operator \( W \in V(X,X)  \) which commutes with
\( F \), where \( 1_{X}(x)=1 \) for \( x \in X \).

\end{corollary}

\begin{corollary}\label{cor:4}

Let \( X \) be a unital Banach algebra and \( B(X) \) be the Banach algebra with
the multiplication operation given in Theorem~\ref{thm:5}. \\ (a) If \( G \in B(X) \)
satisfying \( G^{\ast 2}=G
\ast G=F \) exists and \( F, G \geq 0 \), the commutative unital Banach algebra \( X \)
is an integral domain, and the quasi-product has a left integral domain, then \(
G
\) is unique. \\
(b) Suppose that there exists a square bounded below quasi-product
preserving the positivity on \( X \) (not necessarily commutative) or a square
bounded below quasi-product having a left integral domain with \( X \) being
commutative. For a g-positive operator \( F \in B(X) \), there exists a g-positive
operator \( G \in B(X) \) such that \( G^{\ast 2}=G \ast G=F \) and
\( G \) commutes with any operator \( W \in V(X,X) \) which commutes with
\( F \).

\end{corollary}

\begin{remark}\label{rem:3}

Let \( \tilde{F} \geq 0 \) and \( \tilde{H} \geq 0 \) with \( \tilde{F}(0) \neq 0 \) or
\( \tilde{H}(0)
\neq 0
\), and \( g(0) \neq 0 \). Then given the sufficient conditions in Theorem~\ref{thm:7} and Corollary~\ref{cor:3},
the results also hold for \( \tilde{F} \) and \( \tilde{H} \). Furthermore, given the
sufficient conditions in Corollary~\ref{cor:4} and \( \tilde{F} \in B(X) \), the g-positive
operator
\( G \in B(X) \) defined by
\( G(x)=(||\tilde{F}||_{B(X)}/k)^{1/2}||x||_{X}\hat{G}(x) \) for \( x \neq 0 \)
and
\( G(0)=k^{-1/2}\hat{G}(0) \) satisfies  \( G \ast G=\tilde{F} \), where \( \hat{G}^{2}=\hat{F} \)
and
\( \hat{F}(x)=k\tilde{F}(x)/(||\tilde{F}||_{B(X)}||x||_{X}) \) for \( x \neq 0 \),
\( \hat{F}(0)=k\tilde{F}(0) \), and where \( k \) is some number.

\end{remark}

\subsection{Projection operators}

The spectral resolution of a bounded linear symmetric operator involves both the
projection operator and the spectrum of the operator. The counterparts of
these quantities for a generalized real definite operator are defined and some
basic facts about these quantities are given in this subsection. The first two
lemmas, Lemma~\ref{lem:3} and Lemma~\ref{lem:4}, can be used to prove Lemma~\ref{lem:7}, the main
result of this subsection. Lemma~\ref{lem:7} can be applied to prove the spectral
theorem in next subsection. Note that the proofs of Lemma~\ref{lem:3}, Lemma~\ref{lem:4}, Lemma~\ref{lem:5}, and
Lemma~\ref{lem:6}  in
this subsection are
delegated to the supplementary materials.

The nonlinear projection operator involved in the spectral resolution of the
generalized real definite operator is defined first.

\begin{definition}\label{def:5}

For a subset \( S \) containing \( 0 \) of a normed space \( X \), the
corresponding projection operator \( E_{S}: X \rightarrow X \) is defined by \(
E_{S}(x)=x \) if
\( x \in S \) and
\( E_{S}(x)=0 \) otherwise. As \( X \) is a unital normed algebra,
the projection indicator \( 1_{S}: X \rightarrow X \) is defined by \( 1_{S}(x)=1
\) if \( x \in S \) and
\( 1_{S}(x)=0 \) otherwise.

\end{definition}

 Denote \( N(F) \) as the null space (set) of an operator \( F \), i.e.,
 \( N(F)=\{x: F(x)=0, x \in X\} \).
 The following two
lemmas give the basic properties of the projection operator (indicator) and the
positive and negative parts of \( F \). Let
\( S_{2}
\backslash S_{1} \) denote the intersection of the set
\( S_{2}
\) and the complement of the set \( S_{1} \).

\begin{lemma}\label{lem:3}

Let \( S_{1} \) and \( S_{2} \) both containing \( 0 \) be the subsets of the
normed space
\( X \) . \\
(a) \( ||E_{S_{1}}||_{B(X)} \leq 1 \). \\ (b) The following are equivalent. \\ (i)
\begin{eqnarray*}
E_{S_{1}}\circ E_{S_{2}}=E_{S_{2}}\circ E_{S_{1}}=E_{S_{1}}.
\end{eqnarray*}
(ii) \( S_{1} \subset S_{2} \). \\ (iii) \( N(E_{S_{2}}) \subset N(E_{S_{1}}) \). \\ (iv)
\( ||E_{S_{1}}(x)||_{X} \leq ||E_{S_{2}}(x)||_{X} \) for \( x \in X \). \\ As \( X \)
is a unital normed algebra, the following are equivalent. \\ (i)*
\begin{eqnarray*}
1_{S_{1}}1_{S_{2}}=1_{S_{2}}1_{S_{1}}=1_{S_{1}}.
\end{eqnarray*}
(ii)* \( S_{1} \subset S_{2} \). \\ (iii)* \( N(1_{S_{2}}) \subset N(1_{S_{1}})
\). \\ (iv)* \( ||1_{S_{1}}(x)||_{X} \leq ||1_{S_{2}}(x)||_{X} \) for \( x \in X \). \\
(c) Let \( S_{1} \subset S_{2} \). Then
\( E_{S_{2}-S_{1}}=E_{S_{2}}-E_{S_{1}} \) is
idempotent, i.e.,
\( E_{S_{2}-S_{1}} \circ E_{S_{2}-S_{1}}=E_{S_{2}-S_{1}} \),
and the range of \( E_{S_{2}-S_{1}} \) is
\( S_{2}-S_{1}=(S_{2} \backslash S_{1}) \cup \{0\} \).
As \( X \) is a unital normed algebra, \( 1_{S_{2}- S_{1}}=1_{S_{2}}-1_{S_{1}}
\) satisfies \( 1_{S_{2}- S_{1}}1_{S_{2}-S_{1}}=1_{S_{2}-S_{1}} \).

\end{lemma}

\begin{lemma}\label{lem:4}

Let \( X \) be a unital normed algebra, \( F \) be generalized real definite, and \(
S=N(F^{+})
\).
\\ (a)
\( |F|
\), \( F^{+}
\), and \( F^{-} \) commute with every operator \( W \in V(X,X) \) which commutes with \(
F \). \\ (b) \( F^{+}F^{-}=F^{-}F^{+}=0 \). \\ (c) \( F^{+}1_{S}= 1_{S}F^{+}=0 \) and
\( F^{-}1_{S}=1_{S}F^{-}=F^{-} \). \\ (d) \( F1_{S}=1_{S}F=-F^{-} \) and \(
F(1_{X}-1_{S})=(1_{X}-1_{S})F=F^{+} \).
\\ (e) \( F^{+} \geq 0 \) and \( F^{-} \geq 0 \).

\end{lemma}

The spectrum corresponding to the quasi-product and the g-positivity is
defined as follows.

\begin{definition}\label{def:6}

Let \( F: D(F) \rightarrow X \), where
\( D(F) \subset X \) and \( X \)
is a normed space. Let \( \gamma: D(F) \rightarrow X \) be a g-positive
operator (see Remark~\ref{rem:2}) satisfying
\( [\gamma(x),g(x)]_{X}=k_{1}(x)||x||_{X}||g(x)||_{X} \)
and
\( ||\gamma(x)||_{X}=k_{2}(x)||x||_{X} \)  for \( x \in D(F) \), where both \( k_{1} \) and \( k_{2} \)
are positive bounded functions and are bounded away from 0. The g-resolvent
set of
\( F \), denoted by \( \rho(F) \) and \(  \rho(F) \subset C \),
consists of the scalars \( \lambda \) such that
\( R_{\lambda}=(F-\lambda \gamma)^{-1} \) exists (see \cite{kreyszig}, A1.2), is bounded, and
\( D(R_{\lambda}) \) is a dense set of \( X \). The set
\( \sigma(F)=C \backslash \rho(F) \) is referred to as the g-spectrum of \( F \). As
\( F(x)=\lambda \gamma(x) \) for some
\( x \neq 0 \), \( x \) is referred to as the g-eigenvector of
\( F \) corresponding to the g-eigenvalue \( \lambda \).

\end{definition}

In this article, the generalized real definite operators of interest and the
g-positive operator
\( \gamma \) are in relation to the same quasi-product and
the same operator \( g \).

\begin{remark}\label{rem:4}

The above definition can be also used for the operator \( \tilde{F}
\) with
\( \tilde{F}(0) \neq 0 \). Moreover,
as \( X \) and \( Y \) are Banach spaces and \( G, J \) are in the space of
continuous operators from \( X \) into \( Y \), the classical eigenvalue \(
\lambda
\) of the pair \( (G,J) \) corresponding to the eigenvector \( x \neq 0 \)
satisfies the equation \( G(x)=\lambda J(x) \) (see \cite{appell}, Chapter 9.5,
Chapter 10). As \( X \) and
\( Y \) are normed spaces, the g-eigenvalue \( \lambda \) and g-eigenvector \( x \neq 0 \)
can be defined similarly for the operators
 \( F: D(F) \rightarrow Y \)
and \( \gamma: D(F) \rightarrow Y \) (also see Remark~\ref{rem:2}), i.e., \( F(x)=\lambda
\gamma(x)
\), where
\( D(F) \) is a subset of \( X \).

\end{remark}

The following two lemma, Lemma~\ref{lem:5} and Lemma~\ref{lem:6}, give the results for the
values of the quasi-product of the bounded operators
and the g-eigenvalues of the bounded generalized real
definite operators, respectively.

\begin{lemma}\label{lem:5}

Let \( F \in B(X) \), where \( X \) is a normed space. There exists a positive
number \( \overline{k} \) such that
\begin{eqnarray*}
\left|\left[F(x),g(x)\right]_{X}\right| \leq \overline{k} \left|\left|x\right|\right|_{X}
\left|\left|g(x)\right|\right|_{X}
\end{eqnarray*}
for \( x \in X \).

\end{lemma}

\begin{lemma}\label{lem:6}

Let \( F \) be generalized real definite defined on \( X \), where \( X \) is a
normed space. Then all g-eigenvalues of \( F \), if exist, are real. Further, as \( F
\in B(X) \), all g-eigenvalues, if exist, fall in some bounded interval of \( R \).

\end{lemma}

Note that the results in Lemma~\ref{lem:6} also hold for the generalized real definite
operator
\(
\tilde{F}
\) with
\( \tilde{F}(0) \neq 0 \), i.e., all g-eigenvalues of \( \tilde{F} \), if exist, are real and
fall in some bounded interval of \( R \) if \( \tilde{F} \in B(X) \).

Hereafter denote \( F_{\lambda}=F-\lambda \gamma \) and let
\( 1_{\lambda} \) and \( E_{\lambda} \) be
the projection indicator and the projection operator corresponding to
\( N(F_{\lambda}^{+}) \), respectively, where \( \lambda \in R \). In addition,
let \( \Delta=\mu - \lambda \), \( E_{\Delta}=E_{\mu}-E_{\lambda} \), and
\( 1_{\Delta}=1_{\mu}-1_{\lambda} \), where \( \lambda < \mu \).

\begin{lemma}\label{lem:7}

Let \( X \) be a unital normed algebra. Suppose that
\( F \in B(X) \) is generalized real definite, \( \mu, \lambda \in R \), and \(
\lambda < \mu
\). \\
(a) \( ||1_{\lambda}(x)||_{X} \leq ||1_{\mu}(x)||_{X} \) for \( x \in X \). \\ (b)
\( \lim_{\lambda \rightarrow -\infty}\sup_{x \neq 0,x \in X} 1_{\lambda}(x)=0
\). Further, \( \lim_{\lambda \rightarrow -\infty}E_{\lambda}=0 \),
\( \lim_{\lambda \rightarrow \infty}1_{\lambda}=1_{X} \), and
\( \lim_{\lambda \rightarrow \infty}E_{\lambda}=I \) with respect to the norm
topology \( ||\cdot||_{B(X)} \), where \( I \) is the identity operator defined on
\( X \), i.e.,
\( I(x)=x
\) for \( x \in X \).
\\ (c) \( \lambda \gamma \circ E_{\Delta} \leq F 1_{\Delta} \leq \mu \gamma
\circ E_{\Delta} \).

\end{lemma}

\begin{proof} (a): If  \( x \neq 0 \) and \( 1_{\lambda}(x)=1 \), i.e.,
\( x \in N(F_{\lambda}^{+}) \), \( [F_{\lambda}(x),g(x)]_{X} \leq 0 \) then and thus \( [F(x),g(x)]_{X}
\leq [\lambda \gamma(x),g(x)]_{X} \). Since
\( [\lambda \gamma(x),g(x)]_{X} < [\mu \gamma(x),g(x)]_{X} \),
\( [F(x),g(x)]_{X} < [\mu \gamma(x),g(x)]_{X} \) and
\( [F_{\mu}(x),g(x)]_{X}=[F(x)-\mu \gamma(x),g(x)]_{X} < 0 \) thus. The last
inequality implies
\( F_{\mu}^{+}(x)=0 \) and \( 1_{\mu}(x)=1 \) then. Therefore,
\( ||1_{\lambda}(x)||_{X} \leq ||1_{\mu}(x)||_{X} \) for \( x \in X \). \\
(b): By Lemma~\ref{lem:5}, there exist positive numbers \( \overline{k} \), \( k_{1} \), and
\( k_{2} \) such that
\begin{eqnarray*}
&&\left[F_{\lambda}(x),g(x)\right]_{X} \\ &\geq&
k_{1}\left\{\left[F(x),g(x)\right]_{X}-\left[\lambda
\gamma(x),g(x)\right]_{X}\right\} \\ &\geq&
k_{1}(-\overline{k}-k_{2}\lambda)\left|\left|x\right|\right|_{X}
\left|\left|g(x)\right|\right|_{X}
\\ &>& 0
\end{eqnarray*}
for \( \lambda < -\overline{k}/k_{2} \) and \( x \neq 0 \). This gives that \(
1_{\lambda}(x)=0 \) for
\( \lambda < -\overline{k}/k_{2} \). Similarly, there exist positive numbers \( k^{\ast}_{1} \) and \( k^{\ast}_{2} \)
such that
\begin{eqnarray*}
&&\left[F_{\lambda}(x),g(x)\right]_{X} \\ &\leq&
k^{\ast}_{1}\left\{\left[F(x),g(x)\right]_{X}-\left[\lambda
\gamma(x),g(x)\right]_{X}\right\} \\ &\leq&
k^{\ast}_{1}(\overline{k}-k^{\ast}_{2}\lambda)\left|\left|x\right|\right|_{X}
\left|\left|g(x)\right|\right|_{X}
\\ &<& 0
\end{eqnarray*}
for \( \lambda >\overline{k}/ k^{\ast}_{2} \) and \( x \neq 0 \). This gives that \(
1_{\lambda}=1_{X} \) for \( \lambda > \overline{k}/k^{\ast}_{2} \). The results
for \( E_{\lambda} \) follow because \( E_{\lambda}(x)=1_{\lambda}(x) x \) for \(
x \in X \).
\\ (c): Since \( F_{\mu} 1_{\Delta}=F 1_{\Delta} -\mu \gamma
\circ E_{\Delta} \) and
\( -F_{\mu}1_{\Delta}=-F_{\mu}1_{\mu}(1_{\mu}-1_{\lambda})=
F_{\mu}^{-}(1_{\mu}-1_{\lambda}) \geq 0 \) by Lemma~\ref{lem:3} (b), Lemma~\ref{lem:4} (d), and
Lemma~\ref{lem:4} (e), hence \( F 1_{\Delta} \leq
\mu \gamma \circ E_{\Delta} \). Similarly, because \( F_{\lambda} 1_{\Delta} =F
1_{\Delta} -\lambda \gamma \circ E_{\Delta} \) and
\( F_{\lambda}1_{\Delta}=F_{\lambda}(1_{X}-
1_{\lambda})(1_{\mu}-1_{\lambda}) =F_{\lambda}^{+}(1_{\mu}-1_{\lambda})
\geq 0 \) by Lemma~\ref{lem:3} (b), Lemma~\ref{lem:4} (d), and Lemma~\ref{lem:4} (e), hence \(
F 1_{\Delta} \geq \lambda
\gamma \circ E_{\Delta} \) holds. \end{proof}

\subsection{Spectral theorems}

In this subsection, the spectral resolutions of the generalized real definite
operators in terms of the quasi-product and with respect to some topology are stated in
Theorem~\ref{thm:8} and Theorem~\ref{thm:9}, respectively. Let
\( X \) be a unital Banach algebra in the subsection.

The spectral resolution of interest can be defined based on the lemma below.

\begin{lemma}\label{lem:8}

Let \( F \in B(X) \) be generalized real definite. There exists a bounded interval \(
[m,M] \) with any partition \( \{s_{j}\} \) satisfying
\( m=s_{0}<s_{1}<\cdots<s_{n}=M \) and
\( \Delta_{j}=s_{j}-s_{j-1} < \epsilon_{n} \) such that
\( F_{n}=\sum_{j=1}^{n} \lambda_{j}(\gamma \circ E_{\Delta_{j}}) \)
converges to an operator in \( B(X) \) with respect to the norm
topology \( ||\cdot||_{B(X)} \) and the convergence is independent of
the choice of \(
\lambda_{j}
\in (s_{j-1},s_{j}] \) as \( n \rightarrow\infty \), where \( 1_{\lambda}(x)=0 \)
for \( x \neq 0 \) as \( \lambda=m \), \( 1_{\lambda}=1_{X} \) as \( \lambda=M
\), and
\( 0 < \epsilon_{n}  \mathop{\longrightarrow}\limits_
{n \rightarrow\infty} 0 \).

\end{lemma}

\begin{proof} By Lemma~\ref{lem:3} (b) and Lemma~\ref{lem:7} (b), there exist \( m \) and
\( M \) such that \( 1_{\lambda}(x)=0 \) for \( x \neq 0 \) as \( \lambda \leq m \)
and \( 1_{\lambda}=1_{X} \) as \( \lambda \geq M \). Since for fixed \( x \neq 0
\), \( 1_{m}(x)=0 \), \( 1_{M}(x)=1 \), and
\( 1_{\lambda}(x) \), considered as a function of \(
\lambda \), is right-continuous, there exists
\( \lambda_{x} \in [m,M] \) such that \( 1_{\lambda_{x}}(x)=1 \) and
\( 1_{\lambda}(x)=0 \) for \( \lambda < \lambda_{x} \). Define the operator
\( \tilde{F}(x)=\lambda_{x}\gamma(x) \) as \( x \neq 0 \) and \( \tilde{F}(0)=0 \).
Then \( \tilde{F} \in B(X) \). Since \( \lambda_{x} \in (s_{j-1},s_{j}] \) for some
\( j \) and \( \Delta_{j} < \epsilon_{n} \), hence
\begin{eqnarray*}
&&\left|\left| F_{n}(x)-\tilde{F}(x)\right|\right|_{X} \\ &=& \left|\left|
(\lambda_{j}-\lambda_{x})\gamma(x)\right|\right|_{X} \\ &\leq& \overline{k}
\epsilon_{n}
\left|\left| x\right|\right|_{X}
\end{eqnarray*}
for \( x \in X \), where \( \overline{k} \) is some positive number. \end{proof}

The following definition gives the spectral resolution in terms of operator convergence.

\begin{definition}\label{def:7}

Let \( X \) be a unital Banach algebra, \( F \in B(X) \) be generalized real definite,
and
\(
\{s_{j}\}
\) be any partition of a bounded interval \( [m,M] \) with
\( m=s_{0}<s_{1}<\cdots<s_{n}=M \) and \( \Delta_{j}=s_{j}-s_{j-1}
< \epsilon_{n} \), where
\( 0 < \epsilon_{n}  \mathop{\longrightarrow}\limits_
{n \rightarrow\infty} 0 \). If
\( \sum_{j=1}^{n} \lambda_{j} (\gamma \circ E_{\Delta_{j}}) \) converges to
an operator in the sense of operator
convergence, i.e., with respect to the norm topology \(
||\cdot||_{B(X)} \), and the convergence is independent of the choice of \(
\lambda_{j} \) for \( \lambda_{j} \in (s_{j-1},s_{j}] \) as \( n \rightarrow\infty \),
the limit operator is denoted as \( \int_{m}^{M}\lambda d (\gamma \circ
E_{\lambda}) \). Furthermore, if
\(  f(\lambda_{1})1_{s_{1}}+\sum_{j=2}^{n} f(\lambda_{j}) 1_{\Delta_{j}} \)
converges to an operator in the sense of
operator convergence and the convergence is independent of the choice of \(
\lambda_{j}
\) as \( n \rightarrow\infty \), the limit operator is denoted as \( \int_{m}^{M}
f(\lambda) d 1_{\lambda} \), where \( f: R \rightarrow V(X,X) \) is a mapping
from \( R \) into \( V(X,X) \).

\end{definition}

The following two theorems give the spectral
representations of the generalized real definite operators in terms of the
quasi-product and with respect to the
norm topologies \( ||\cdot||_{B(X)} \), respectively. Moreover, \(
[F(x),g(x)]_{X}, x \in X \), can be expressed as an ordinary Riemann-Stieltjes
integral in terms of a certain equivalence relation. The involved equivalence
relation
\( \equiv \) for two functionals
\( G_{1} \) and \( G_{2} \) from \( X \) into \( K \) is denoted as
\( G_{1} \equiv G_{2} \) (or \( G_{1}(x) \equiv G_{2}(x) \) for convenience) if and only if
there exist positive numbers \( \underline{k} \) and \( \overline{k} \) such that
\begin{eqnarray*}
\underline{k} \left|G_{2}(x)\right|
\leq \left|G_{1}(x)\right|
\leq \overline{k} \left|G_{2}(x)\right|
\end{eqnarray*}
for all \( x \in X \).

\begin{theorem}\label{thm:8}

Let \( F \in B(X) \) be generalized real definite. Then
\begin{eqnarray*}
\left[F(x),g(x)\right]_{X}=\left[\left[\int_{m}^{M}\lambda d (\gamma \circ
E_{\lambda})\right](x),g(x)\right]_{X}
\end{eqnarray*}
and
\begin{eqnarray*}
\left[F(x),g(x)\right]_{X} \equiv \int_{m}^{M}\lambda d w_{x}(\lambda)
\end{eqnarray*}
for \( x \in X \), where \( [m,M] \) is some bounded interval depending on \( F
\),
\( w_{x}(\lambda)=[(\gamma \circ E_{\lambda})(x),g(x)]_{X} \),
and the second integral is the
ordinary Riemann-Stieltjes integral.

\end{theorem}

\begin{proof} \( \int_{m}^{M}\lambda d (\gamma \circ E_{\lambda}) \) exists
by Lemma~\ref{lem:8}, where \( 1_{m}(x)=0 \) for \( x \neq 0 \) and \( 1_{M}=1_{X} \).
Note that \( \sum_{j=1}^{n}\Delta_{j}=M-m \) and \(
\sum_{j=1}^{n}1_{\Delta_{j}}(x)=1 \) for \( x \neq 0 \). Then
\( F=F\sum_{j=1}^{n}1_{\Delta_{j}}=\sum_{j=1}^{n}F1_{\Delta_{j}} \).
By Lemma~\ref{lem:7} (c) and Lemma~\ref{lem:1} (b), (c), and (d),
\begin{eqnarray*}
\sum_{j=0}^{n-1}s_{j}(\gamma \circ E_{\Delta_{j+1}}) \leq F \leq \sum_{j=1}^{n}
s_{j}(\gamma \circ E_{\Delta_{j}})
\end{eqnarray*}
and hence
\begin{eqnarray*}
0 \leq \sum_{j=1}^{n}s_{j}(\gamma \circ E_{\Delta_{j}}) -F
\leq \sum_{j=1}^{n}\Delta_{j}(\gamma \circ E_{\Delta_{j}})
\leq \epsilon_{n} \sum_{j=1}^{n} \gamma \circ E_{\Delta_{j}}.
\end{eqnarray*}
Then by Lemma~\ref{lem:1} (a),
\begin{eqnarray*}
&& \left[\sum_{j=1}^{n}s_{j}(\gamma \circ
E_{\Delta_{j}})(x)-F(x),g(x)\right]_{X} \\ &\leq& \left[\epsilon_{n}
\sum_{j=1}^{n} (\gamma \circ E_{\Delta_{j}})(x),g(x)\right]_{X} \\ &=&
\left[\epsilon_{n} \gamma(x),g(x)\right]_{X} \\ &\leq& \overline{k} \epsilon_{n}
\left|\left|x\right|\right|_{X}\left|\left|g(x)\right|\right|_{X}
\end{eqnarray*}
for \( x \in X \), where \( \overline{k} \) is some positive number. \(
[[\int_{m}^{M}\lambda d (\gamma
\circ E_{\lambda})-F](x),g(x)]_{X}=0 \) by the continuity of the quasi-product
and \( [[\int_{m}^{M}\lambda d (\gamma \circ
E_{\lambda})](x),g(x)]_{X}=[F(x),g(x)]_{X} \) thus. Since by property (c) of the
quasi-product,
\begin{eqnarray*}
\left[\sum_{j=1}^{n}\lambda_{j}(\gamma \circ E_{\Delta_{j}})(x),g(x)\right]_{X}
\equiv
\sum_{j=1}^{n}\lambda_{j}\left[(\gamma \circ E_{\Delta_{j}})(x),g(x)\right]_{X}
\end{eqnarray*}
for \( x \in X \),
\begin{eqnarray*}
&&\int_{m}^{M}\lambda d w_{x}(\lambda) \\ &\equiv&
\lim_{n \rightarrow \infty} \sum_{j=1}^{n}\lambda_{j}\left[(\gamma \circ
E_{\Delta_{j}})(x),g(x)\right]_{X} \\ &\equiv& \lim_{n \rightarrow \infty}\left[
\sum_{j=1}^{n}\lambda_{j}(\gamma
\circ E_{\Delta_{j}})(x),g(x)\right]_{X} \\ &=&\left[F(x),g(x)\right]_{X}
\end{eqnarray*}
thus. \end{proof}

In the following, the spectral resolution of the generalized real definite
operator in terms of operator convergence is given. The key sufficient
condition for the operator convergence in this subsection and in Section 4.1 is
defined below.

\begin{definition}\label{def:8}

The uniform spectral representation condition on a unital Banach algebra \( X \)
is as follows: There exists a quasi-product
\\ (a)  which  is square bounded below and preserves the positivity  \\
or \\
(b)  has the left integral domain.

\end{definition}

\begin{theorem}\label{thm:9}

Let \( F \in B(X) \) be generalized real definite. If the uniform spectral
representation condition on \( X \) holds, then \( F \) has the spectral
representation
\begin{eqnarray*}
F=\int_{m}^{M}\lambda d (\gamma \circ E_{\lambda}),
\end{eqnarray*}
where \( [m,M] \) is some bounded interval depending on \( F \).

\end{theorem}

\begin{proof} \( \int_{m}^{M}\lambda d (\gamma \circ E_{\lambda}) \) exists
by Lemma~\ref{lem:8}. As the quasi-product has the left integral domain, i.e., the uniform spectral representation condition
(b) satisfied, the equation \( [[\int_{m}^{M}\lambda
d (\gamma \circ E_{\lambda})-F](x),g(x)]_{X}=0 \) implies that \(
F(x)=[\int_{m}^{M}\lambda d (\gamma \circ E_{\lambda})](x) \) for every \( x
\in X \). Then by  Lemma~\ref{lem:8}, \( F= \int_{m}^{M}\lambda d (\gamma \circ E_{\lambda})  \).

Next, assume that the uniform spectral representation condition (a) holds.
Let \( F_{n}=\sum_{j=1}^{n} s_{j} (\gamma \circ E_{\Delta_{j}}) \).
Then
\begin{eqnarray*}
0 \leq F_{n}-F \leq \epsilon_{n} \sum_{j=1}^{n} \gamma \circ E_{\Delta_{j}}
\end{eqnarray*}
and thus
\begin{eqnarray*}
0 \leq \left(F_{n}-F\right)^{2} \leq \epsilon_{n}^{2}\left (\sum_{j=1}^{n}
\gamma \circ E_{\Delta_{j}}\right)^{2}
\end{eqnarray*}
by Theorem~\ref{thm:7}. Since the quasi-product is square bounded below, then there
exist positive numbers \( k_{1} \) and \( k_{2} \) such that for \( x \in X \),
\begin{eqnarray*}
&&
k_{1}\left|\left|F_{n}(x)-F(x)\right|\right|^{2}_{X}\left|\left|g(x)\right|\right|_{X}
\\ &\leq& \left[\left(F_{n}-F\right)^{2}(x),g(x)\right]_{X} \\ &\leq&
\left[\epsilon_{n}^{2} \left(\sum_{j=1}^{n} \gamma \circ
E_{\Delta_{j}}\right)^{2}(x),g(x)\right]_{X} \\ &\leq& k_{2}
\epsilon_{n}^{2}\left|\left|x\right|\right|^{2}_{X}\left|\left|g(x)\right|\right|_{X}.
\end{eqnarray*}
Hence \( \left|\left|F_{n}-F\right|\right|_{B(X)}
\mathop{\longrightarrow}\limits_ {n \rightarrow\infty} 0 \) holds.
\end{proof}

\begin{remark}\label{rem:5}

Theorem~\ref{thm:8} and Theorem~\ref{thm:9}  under only the uniform spectral representation
condition (b) also hold as
\( X \) is a Banach space, i.e., the condition for \( X \) being relaxed. The key is to
use \( F\circ E_{S} \) in place of \( F 1_{S} \) in the associated results, where
\( S \) containing 0 is a subset of \( X \).

\end{remark}

\section{Extensions}

In this section, the operational calculus of the generalized real definite operators is
given in the first subsection. Furthermore, the bounded generalized real definite
operators with the spectral representation in the sense of operator convergence
turns out to be associated with a class of operators defined in the second
subsection.

\subsection{Operational calculus}

In this subsection, the extensions of the spectral theorem in the previous
subsection to polynomial functions and continuous functions are stated in two
theorems, Theorem~\ref{thm:10} and Theorem~\ref{thm:12}, respectively. Let \( X \) be a unital
Banach algebra as in Section 3.3. The polynomial of the
generalized real definite operator is defined below. Note that the proofs of the lemmas in
this subsection are
delegated to the supplementary materials.

\begin{definition}\label{def:9}

 Let \( B(X) \) be a unital Banach algebra with
 the multiplication operation \( * \) given in Theorem~\ref{thm:5}.
Let \( p(\lambda)=\sum_{i=0}^{n}a_{i}\lambda^{i} \) be a polynomial in \(
\lambda \) with real coefficients \( a_{i} \), i.e., \( p \) being
over the real field. Then
\( p(F)=\sum_{i=0}^{n}a_{i}F^{*i} \), where \( F \in B(X) \)  and \( F^{*0}=e \).

\end{definition}

\begin{lemma}\label{lem:9}

Let \( F_{1n}, F_{1}, F_{2n}, F_{2} \in B(X) \). If
\( F_{1n} \mathop{\longrightarrow}\limits_
{n \rightarrow\infty} F_{1} \) and \( F_{2n} \mathop{\longrightarrow}\limits_ {n
\rightarrow\infty} F_{2} \) in the sense of operator convergence, then \( F_{1n} \ast F_{2n}
\mathop{\longrightarrow}\limits_ {n
\rightarrow\infty} F_{1} \ast F_{2} \) in the sense of operator
convergence.

\end{lemma}

\begin{theorem}\label{thm:10}

Let \( F \in B(X) \) be generalized real definite. Then \( p(F) \) has the spectral
representation
\begin{eqnarray*}
p(F)=\int_{m}^{M}p(\lambda \gamma) d 1_{\lambda}
\end{eqnarray*}
if the uniform spectral
representation condition on \( X \) holds.

\end{theorem}

\begin{proof} Since for \( k > j \),
\begin{eqnarray*}
1_{\Delta_{j}}1_{\Delta_{k}}=\left(1_{s_{j}}-1_{s_{j-1}}\right)\left(1_{s_{k}}-
1_{s_{k-1}}\right)= 1_{s_{j}}-1_{s_{j}}-1_{s_{j-1}}+1_{s_{j-1}}=0
\end{eqnarray*}
by Lemma~\ref{lem:3} (b) and thus
\( (\gamma \circ E_{\Delta_{j}})(\gamma \circ E_{\Delta_{k}})=
\gamma^{2}1_{\Delta_{j}}1_{\Delta_{k}}=0 \).
Therefore,
\begin{eqnarray*}
\left[\sum_{j=1}^{n} \lambda_{j} (\gamma \circ E_{\Delta_{j}})\right]^{\ast k}
=\sum_{j=1}^{n} \lambda^{k}_{j} (\gamma \circ E_{\Delta_{j}})^{\ast k}
=\sum_{j=1}^{n} (\lambda_{j}\gamma)^{\ast k}1_{\Delta_{j}},
\end{eqnarray*}
where \( k \) is a nonnegative integer. By the condition imposed on the
quasi-product and then using Theorem~\ref{thm:9},
\( \sum_{j=1}^{n} \lambda_{j} (\gamma \circ E_{\Delta_{j}})
\mathop{\longrightarrow}\limits_
{n \rightarrow\infty} F \) in the sense of operator
convergence. By Lemma~\ref{lem:9},
\begin{eqnarray*}
\sum_{j=1}^{n} (\lambda_{j}\gamma)^{\ast k}1_{\Delta_{j}}=
\left[\sum_{j=1}^{n} \lambda_{j} (\gamma \circ E_{\Delta_{j}})\right]^{\ast k}
\mathop{\longrightarrow}\limits_
{n \rightarrow\infty} F^{\ast k}
\end{eqnarray*}
and hence
\begin{eqnarray*}
p( \lambda_{1} \gamma) 1_{s_{1}}+\sum_{j=2}^{n} p( \lambda_{j} \gamma)
1_{\Delta_{j}}= p\left[\sum_{j=1}^{n} \lambda_{j} (\gamma \circ
E_{\Delta_{j}})\right]
\mathop{\longrightarrow}\limits_
{n \rightarrow\infty} p(F),
\end{eqnarray*}
i.e., \( p(F)=\int_{m}^{M}p(\lambda \gamma) d 1_{\lambda} \) in the sense of
operator convergence. \end{proof}

For a bounded linear self-adjoint operator \( T \), the normed values of
\( p(T) \) is not greater than the normed value of  \( p(\lambda) \), where
\( p(\lambda) \) is considered as an element in the space of all continuous functions
defined on some compact interval. The following lemma can be considered as
the counterpart of the one for the bounded linear self-adjoint operator. Let \(
C([m,M] \) be the Banach space of all real-valued continuous functions defined on \(
[m,M]
\) with the supremum norm.

\begin{lemma}\label{lem:10}

Let \( F \in B(X) \) be generalized real definite. If the uniform spectral
representation condition on \( X \) holds, then there exist a positive number \(
\overline{k} \) and a bounded interval \( [m,M] \) depending on
\( F \) such that for any polynomial
function \( p \) with real coefficients,
\begin{eqnarray*}
\left|\left|p(F)\right|\right|_{B(X)} \leq \overline{k}
\max_{\lambda \in [m,M]}\left|p(\lambda)\right|=\overline{k}
\left|\left|p\right|\right|_{C([m,M])}.
\end{eqnarray*}

\end{lemma}

For any \( f \in C([m,M]) \), the operator \( f(F) \) and its spectral theorem can
be defined and established based on Lemma~\ref{lem:10}. The following theorem gives
the existence of the limit operator of a sequence of polynomials of the
generalized real definite operator.

\begin{theorem}\label{thm:11}

Let \( F
\in B(X) \) be generalized real definite. If the uniform spectral representation
condition on \( X \) holds, then there exists a bounded interval \( [m,M] \)
depending on \( F \) such that \( \{p_{n}(F)\} \) converges to an operator in \(
B(X) \) in the norm \( ||\cdot||_{B(X)} \), where
\( \{p_{n}\} \) is any convergent sequence of
polynomial functions with real coefficients in the space \( C([m,M]) \).

\end{theorem}

\begin{proof} By Lemma~\ref{lem:10}, for any \( \epsilon > 0 \), there exist a
positive number \( N \)
 such that for \( n, l > N \),
\begin{eqnarray*}
\left|\left|p_{n}(F)-p_{l}(F)\right|\right|_{B(X)} \leq
\overline{k}\left|\left|p_{n}-p_{l}\right|\right|_{C([m,M])} <  \epsilon
\end{eqnarray*}
and thus \( \{p_{n}(F)\} \) is a Cauchy sequence in the Banach space \( B(X) \),
where
 \( \overline{k} \) is some positive number.
\end{proof}

According to the above theorem, the operator corresponding to any \( f
\in C([m,M]) \) and the generalized real definite operator \( F \in B(X) \) can be
defined as follows.

\begin{definition}\label{def:10}

Let \( B(X) \) be a unital Banach algebra. If there exists a sequence of
polynomial functions \( \{p_{n}\} \) defined on
\( [m,M] \) over the real field converging uniformly to \( f \) in \( C([m,M]) \), i.e.,
the convergence being in the norm \( ||\cdot||_{C([m,M])} \),
and the limit of the sequence of operators \(
\{p_{n}(F)\}
\) in the norm \( ||\cdot||_{B(X)}
\) corresponding to  the generalized real definite operator
\( F \in B(X) \) exists, then the limit in \( B(X) \) is denoted by \( f(F) \).

\end{definition}

As the uniform spectral representation condition on
\( X \) holds, \( f(F) \) exists and is unique, i.e., \( f(F) \) being
the limit corresponding to any sequence of polynomial functions converging to
\( f \) in \( C([m,M]) \), as indicated by the corollary below.

\begin{corollary}\label{cor:6}

Let \( F \in B(X) \) be generalized real definite. If the uniform spectral
representation condition on
\( X \) holds, then \( f(F) \) exists for any \( f \in C([m,M]) \)
and is independent of the choice of the sequence of
polynomial functions in
\( C([m,M]) \), i.e.,
\( f(F) \) being the limit of any sequence of operators \( \{p_{n}(F)\} \) satisfying
that \( \{p_{n}\} \) defined on
\( [m,M] \) over the real field converges uniformly to \( f \) in \( C([m,M]) \),
where \( [m,M] \) is some bounded interval depending on
\( F \).

\end{corollary}

\begin{proof} Because there exists a sequence of polynomial functions converging
uniformly to \( f \) in \( C([m,M]) \) by Weierstrass theorem, \( f(F) \) exists by
Theorem~\ref{thm:11} and Definition~\ref{def:10}. Next, let \( \{p_{n}\} \) and \( \{p_{n}^{\ast}\} \)
be the sequences of polynomial functions both converging uniformly to \( f \)
in \( C([m,M]) \). Then \( p_{n}(F)
\mathop{\longrightarrow}\limits_ {n \rightarrow\infty} f(F) \) and \(
p^{\ast}_{n}(F)
\mathop{\longrightarrow}\limits_ {n \rightarrow\infty} f^{\ast}(F) \) by Theorem~\ref{thm:11}.
Furthermore, because
\begin{eqnarray*}
&&\left|\left|f(F)-f^{\ast}(F)\right|\right|_{B(X)} \\ &\leq&
\left|\left|f(F)-p_{n}(F)\right|\right|_{B(X)} +\left|\left|p_{n}(F)-p_{n}^{\ast}(F)
\right|\right|_{B(X)} +
\left|\left|p_{n}^{\ast}(F)-f^{\ast}(F)\right|\right|_{B(X)}
\end{eqnarray*}
and
\begin{eqnarray*}
\left|\left|p_{n}(F)-p^{\ast}_{n}(F)\right|\right|_{B(X)} \leq
\overline{k}\left|\left|p_{n}-p^{\ast}_{n}\right|\right|_{C([m,M])},
\end{eqnarray*}
then letting \( n \rightarrow \infty \) gives \( f(F)=f^{\ast}(F) \) .
\end{proof}

The following lemma is an extension of Lemma~\ref{lem:10} to the continuous functions
and can be used to prove the spectral theorem corresponding to the
continuous functions.

\begin{lemma}\label{lem:11}

 Let \( F \in B(X) \) be generalized real definite. If the uniform spectral
 representation condition on \( X \) holds, then
there exist a positive number \( \overline{k} \) and a bounded interval \( [m,M]
\) depending on
\( F \) such that for any \( f \in C([m,M]) \),
\begin{eqnarray*}
\left|\left|f(F)\right|\right|_{B(X)} \leq \overline{k}
\max_{\lambda \in [m,M]}\left|f(\lambda)\right|=  \overline{k}
\left|\left|f\right|\right|_{C([m,M])}.
\end{eqnarray*}

\end{lemma}

\begin{theorem}\label{thm:12}

If the uniform spectral representation condition on \( X \) holds, then for a
generalized real definite operator \( F
\in B(X) \) and any \( f \in C([m,M]) \) , \( f(F) \) has the spectral representation
\begin{eqnarray*}
f(F)=\int_{m}^{M}f(\lambda \gamma) d 1_{\lambda},
\end{eqnarray*}
 where \( [m,M] \) is some bounded interval depending on \( F \).

\end{theorem}

\begin{proof} There exists a bounded interval \( [m,M] \) and
a positive number \( \bar{k} \) such that
\( ||f(\lambda \gamma)||_{B(X)} \leq \overline{k} ||f||_{C([m,M])} \) for
any \( f \in C([m,M]) \) and \( \lambda \in [m^{\ast},M^{\ast}] \) by Lemma~\ref{lem:11}
with \( 1_{m^{\ast}}(x)=0
\) for \( x \neq 0
\), \( 1_{M^{\ast}}=1 \), and
\( [m^{\ast},M^{\ast}] \subset [m,M] \). Let
\( \{p_{l}\} \) be a sequence of polynomial functions defined on \( [m,M] \) over the real field
converging uniformly to \( f \) in \( C([m,M]) \). By Definition~\ref{def:10}, Theorem~\ref{thm:10},
Corollary~\ref{cor:1}, and Lemma~\ref{lem:11}, for every \(
\epsilon > 0 \), there exist positive numbers \( N \) and \( N_{l}  \) such that for \( n > N_{l} \) and
\( l > N \) ,
\( ||p_{l}(F)-f(F)||_{B(X)} \leq \epsilon/3 \),
\( ||p_{l}(\lambda_{1}\gamma)1_{s_{1}}+\sum_{j=2}^{n}
p_{l}(\lambda_{j}\gamma)1_{\Delta_{j}}-p_{l}(F)||_{B(X)} \leq \epsilon/3 \),
and
\begin{eqnarray*}
&&\left|\left|f(\lambda_{1}\gamma)1_{s_{1}}+\sum_{j=2}^{n}
f(\lambda_{j}\gamma)1_{\Delta_{j}}-
p_{l}(\lambda_{1}\gamma)1_{s_{1}}-\sum_{j=2}^{n}
p_{l}(\lambda_{j}\gamma)1_{\Delta_{j}}\right|\right|_{B(X)}
\\ &\leq& \max_{m^{\ast} \leq \lambda \leq M^{\ast}} \left|\left|f(\lambda
\gamma)-p_{l}(\lambda \gamma)\right|\right|_{B(X)} \\ &\leq& \overline{k}
\left|\left|f-p_{l}\right|\right|_{C([m,M])} \\ &\leq& \epsilon/3.
\end{eqnarray*}
Therefore,
\begin{eqnarray*}
&&
\left|\left|f(\lambda_{1}\gamma)1_{s_{1}}+\sum_{j=2}^{n}
f(\lambda_{j}\gamma)1_{\Delta_{j}}-f(F)\right|\right|_{B(X)}
\\ &\leq&
\left|\left|f(\lambda_{1}\gamma)1_{s_{1}}+\sum_{j=2}^{n}
f(\lambda_{j}\gamma)1_{\Delta_{j}}-
p_{l}(\lambda_{1}\gamma)1_{s_{1}}-\sum_{j=2}^{n}
p_{l}(\lambda_{j}\gamma)1_{\Delta_{j}}\right|\right|_{B(X)}
\\
&+&\left|\left|p_{l}(\lambda_{1}\gamma)1_{s_{1}}+\sum_{j=2}^{n}
p_{l}(\lambda_{j}\gamma)1_{\Delta_{j}}-p_{l}(F)\right|\right|_{B(X)} +
\left|\left|p_{l}(F)-f(F)\right|\right|_{B(X)} \\ &\leq& \epsilon
\end{eqnarray*}
and the result holds.
\end{proof}

\begin{remark}\label{rem:6}

For the generalized real definite operator \( \tilde{F} \in B(X) \) with
\( \tilde{F}(0) \neq 0 \), i.e., \( \tilde{F}=F+\tilde{F}(0)e \) and
the projection indicator \( 1_{\lambda} \) corresponding to the null space of
\( F_{\lambda}^{+} \), the spectral resolution of the polynomial
\( p(\tilde{F}) \) given the condition in Theorem~\ref{thm:10} is
\begin{eqnarray*}
p(\tilde{F})=\int_{m}^{M}p\left[\lambda \gamma+\tilde{F}(0)e\right] d 1_{\lambda}.
\end{eqnarray*}
Similarly, the spectral integral of the operator \( f(\tilde{F}) \) given the
condition in Theorem~\ref{thm:12} is
\begin{eqnarray*}
f(\tilde{F})=\int_{m}^{M}f\left[\lambda \gamma+\tilde{F}(0)e\right] d 1_{\lambda}.
\end{eqnarray*}

\end{remark}

\subsection{Nonlinear spectral operators}

In \cite{dunford2}, the theory of the linear spectral operators has been discussed
thoroughly. In this subsection, the nonlinear spectral operators based on the
projection operator given in Definition~\ref{def:5} are defined and a basic result, Theorem~\ref{thm:13},
is given.

\begin{definition}\label{def:11}

Let \( X \)  be a normed space. A spectral projection \( E \) on \( (m,M] \) is an operator-value function from the subsets
\( \cup_{i=1}^{n}(a_{i},b_{i}] \) of \( (m,M] \) into \( B(X) \),\ \( m,M \in R,  (a_{i},b_{i}] \subset (m,M] \), with the following properties. \\
(a) \( E\{(a_{i},b_{i}]\} \) is a projection operator, i.e., \( E\{(a_{i},b_{i}]\}(x)=x \) if
\( x \in S \) and
\( E\{(a_{i},b_{i}]\}(x)=0 \) otherwise, where \( S \) containing \( 0 \) is some subset of \( X \). \\
(b) \( E(\phi)=0 \) and \( E\{(m,M]\}=I \). \\
(c)
\begin{eqnarray*}
E\{(a_{1},b_{1}] \cap (a_{2},b_{2}]\}= E\{(a_{1},b_{1}]\}  \circ E\{(a_{2},b_{2}] \}=E\{(a_{2},b_{2}]\}  \circ E\{(a_{1},b_{1}] \}.
\end{eqnarray*}
In addition, if \( (a_{1},b_{1}] \cap (a_{2},b_{2}]=\phi \), then
\begin{eqnarray*}
E\{(a_{1},b_{1}] \cup (a_{2},b_{2}]\}= E\{(a_{1},b_{1}]\}  +  E\{(a_{2},b_{2}] \} .
\end{eqnarray*}

\end{definition}

The existence of the spectral operators of interest is due to the following lemma. The proof of this lemma
 is analogous to Lemma~\ref{lem:8} and is not presented.

\begin{lemma}\label{lem:12}

Let \( X \) be a Banach space and \( \{s_{j}\} \) be a partition of a bounded interval \( [m,M] \) with \( m=s_{0} < s_{1} < \cdots < s_{n} = M \),
\( s_{j}-s_{j-1} < \epsilon_{n} \) and \( \epsilon_{n} \mathop{\longrightarrow}\limits_ {n \rightarrow\infty} 0 \). Then,
\( \sum_{j=1}^{n}f(\lambda_{j})E\{(s_{j-1},s_{j}]\} \) converges to an operator in \( B(X) \) with respect to the norm
topology \( ||\cdot||_{B(X)} \) and the convergence is  independent of the choice of the points \( \lambda_{j} \in
(s_{j-1},s_{j}] \) as \( n \rightarrow \infty \), where
 \( f \in C([m,M]) \).

\end{lemma}

Based on the above lemma, the resulting limit operator can be defined.

\begin{definition}\label{def:12}

Let \( X \) be a Banach space and
\( \{s_{j}\} \) be a partition of a bounded interval \( [m,M] \) with \( m=s_{0} < s_{1} < \cdots < s_{n} = M \),
\( s_{j}-s_{j-1} < \epsilon_{n} \) and \( \epsilon_{n}\mathop{\longrightarrow}\limits_ {n \rightarrow\infty}  0 \). Then, the limit
operator \( F \) of  \( \sum_{j=1}^{n}f(\lambda_{j})E\{(s_{j-1},s_{j}]\} \) as \( n \rightarrow \infty \) with respect to the norm
topology
\( ||\cdot||_{B(X)} \) is denoted as
\begin{eqnarray*}
F=\int_{m}^{M}f(\lambda)d E,
\end{eqnarray*}
where \( \lambda_{j} \in  (s_{j-1},s_{j}] \)  and \( f \in C([m,M]) \). The operator \( F \) is referred to as the  nonlinear spectral operator with respect to
the spectral projection  \( E \) on \( (m,M] \) and the function \( f \). The class of nonlinear
spectral operators with respect to the spectral projection  \( E \) on \( (m,M] \) and any function \( f \in C([m,M]) \) is denoted as
\( S_{E,C([m,M])}(X) \).

\end{definition}

\begin{theorem}\label{thm:13}

\(  S_{E,C([m,M])}(X) \) is a subspace of \( B(X) \), where \( X \) is a Banach space. Further, If \( E\{(m,\bar{\lambda}]\}- E\{(m,\underline{\lambda}]\} \neq 0 \) for
any \( \underline{\lambda},\bar{\lambda} \in (m,M] \) and \( \underline{\lambda} < \bar{\lambda} \), then \(  S_{E,C([m,M])}(X) \)  is a Banach space.

\end{theorem}

\begin{proof}

\( S_{E,C([m,M])}(X)  \)  is a subspace of \( B(X) \) because for the operators \( F_{1},F_{2}  \in S_{E,C([m,M])}(X) \) corresponding to the
functions \( f_{1} \) and \( f_{2} \), respectively, and \( \alpha \in K \),
\begin{eqnarray*}
&& (\alpha F_{1}+F_{2})\\
&=&  \lim_{n \rightarrow \infty} \sum_{j=1}^{n}\alpha f_{1}(\lambda_{j}) E\left\{(s_{j-1},s_{j}]\right\}\\
&&+
 \lim_{n \rightarrow \infty} \sum_{j=1}^{n} f_{2}(\lambda_{j})E\left\{(s_{j-1},s_{j}]\right\}\\
 &=& \lim_{n \rightarrow \infty} \sum_{j=1}^{n}\left[\alpha f_{1}(\lambda_{j}) +f_{2}(\lambda_{j})\right]E\left\{(s_{j-1},s_{j}]\right\} \\
 &=& \int_{m}^{M}\left[\alpha f_{1}(\lambda)+ f_{1}(\lambda)\right]dE,
 \end{eqnarray*}
 and \( \alpha f_{1} +f_{2} \in C([m,M]) \).

 To prove the completeness of \(  S_{E,C([m,M])}(X) \), let \( F_{n}   \mathop{\longrightarrow}\limits_ {n \rightarrow\infty} F \), where
 \( F_{n} \in S_{E,C([m,M])}(X) \) and \( F \in B(X) \). Then \( \{F_{n}\} \) is Cauchy. Further, if
 the required condition holds, there exists a \( x_{\lambda} \) depending on \( \lambda \) such that for any \( \lambda \in (m,M] \) ,
 \begin{eqnarray*}
 &&\left|\left|\left\{\sum_{j=1}^{n}\left[f_{l}(\lambda_{j})-f_{m}(\lambda_{j})\right] E\left\{(s_{j-1},s_{j}]\right\}
 \right\}(x_{\lambda})\right|\right|_{X}  \\
 &=& \left|\left|\left[f_{l}(\lambda)-f_{m}(\lambda)\right]x_{\lambda}\right|\right|_{X} \\
 &=&\left|f_{l}(\lambda)-f_{m}(\lambda)\right|\left|\left|x_{\lambda}\right|\right|_{X},
 \end{eqnarray*}
 where \( f_{l} \) and \( f_{m} \) are the functions in \( C([m,M]) \) corresponding to the
 spectral operators \( F_{l} \) and \( F_{m} \), respectively.
Then by letting \( n \rightarrow \infty \) in the above equation,
 for any \( \epsilon > 0 \), there exists a positive integer \( N \) such
 that for \( l,m > N \),
 \begin{eqnarray*}
 \left|\left|f_{l}-f_{m}\right|\right|_{C([m,M])} \leq \epsilon.
 \end{eqnarray*}
 Hence \( \{f_{n}\} \) is Cauchy and there exists a function
 \( f  \in C([m,M]) \) such that \( f_{n} \mathop{\longrightarrow}\limits_ {n \rightarrow\infty}  f \) with respect to \( ||\cdot||_{C([m,M])} \)
owing to the completeness of \( C([m,M]) \). Then,
 \( F_{n} \mathop{\longrightarrow}\limits_ {n \rightarrow\infty}  \int_{m}^{M}fdE \) because for \( x \neq 0 \),
 \begin{eqnarray*}
&& \left|\left|\left[F_{n}-\int_{m}^{M}f(\lambda)dE\right](x)\right|\right|_{X}  \\
&=& \left|\left|\left\{\int_{m}^{M}\left[f_{n}(\lambda)-f(\lambda)\right]dE\right\}(x)\right|\right|_{X} \\
&\leq& \left|\left|f_{n}-f\right|\right|_{C([m,M])}\left|\left|x\right|\right|_{X}.
\end{eqnarray*}
Therefore, \( F=\int_{m}^{M}fdE \) and \( F \in S_{E,C([m,M])}(X) \), i.e., \( S_{E,C([m,M])}(X)  \) being closed.
\end{proof}

\begin{remark}\label{rem:7}

As \( \gamma=I \), the bounded generalized real definite operators with the spectral representation
fall in \( S_{E,C([m,M])}(X) \). Note that a more general class of nonlinear spectral
operators can be defined as the limiting operators of  the operators \( \sum_{j=1}^{n}f(\lambda_{j})(r\circ E\{(s_{j-1},s_{j}]\}) \)
as \( n \rightarrow \infty \) with respect to the norm
topology
\( ||\cdot||_{B(X)} \) and can be denoted as
\begin{eqnarray*}
F=\int_{m}^{M}f(\lambda)d (r \circ E).
\end{eqnarray*}
Moreover, Lemma~\ref{lem:12} and Theorem~\ref{thm:13} can be generalized for the class of the nonlinear spectral operators.

\end{remark}

\end{document}